\documentclass[11pt, a4paper, twoside]{article}
\usepackage{amsfonts}
\usepackage{mathrsfs}
\usepackage{amsmath, amsthm}
\usepackage{pgf,tikz}
\usepackage{amssymb}
\usepackage{fancyhdr}
\usepackage[colorlinks, linkcolor=black, anchorcolor=black, citecolor=black]{hyperref}

\numberwithin{equation}{section}

\allowdisplaybreaks

\begin{document}

\fancyhf{}

\fancyhead[OR]{\thepage}

\renewcommand{\headrulewidth}{0pt}
\renewcommand{\thefootnote}{\fnsymbol {footnote}}

\theoremstyle{plain} 
\newtheorem{thm}{\indent\sc Theorem}[section] 
\newtheorem{lem}[thm]{\indent\sc Lemma}
\newtheorem{cor}[thm]{\indent\sc Corollary}
\newtheorem{prop}[thm]{\indent\sc Proposition}
\newtheorem{claim}[thm]{\indent\sc Claim}
\theoremstyle{definition} 
\newtheorem{dfn}[thm]{\indent\sc Definition}
\newtheorem{rem}[thm]{\indent\sc Remark}
\newtheorem{ex}[thm]{\indent\sc Example}
\newtheorem{notation}[thm]{\indent\sc Notation}
\newtheorem{assertion}[thm]{\indent\sc Assertion}
%
%
\numberwithin{equation}{section}
\renewcommand{\proofname}{\indent\sc Proof.} 
\def\C{\mathbb{C}}
\def\R{\mathbb{R}}
\def\Rn{{\mathbb{R}^n}}
\def\M{\mathbb{M}}
\def\N{\mathbb{N}}
\def\Q{{\mathbb{Q}}}
\def\Z{\mathbb{Z}}
\def\F{\mathcal{F}}
\def\L{\mathcal{L}}
\def\S{\mathcal{S}}
\def\supp{\operatorname{supp}}
\def\essi{\operatornamewithlimits{ess\,inf}}
\def\esss{\operatornamewithlimits{ess\,sup}}
\def\dlim{\displaystyle\lim}

\fancyhf{}

\fancyhead[EC]{W. LI, H. Wang}

\fancyhead[EL]{\thepage}

\fancyhead[OC]{Maximal functions associated with nonisotropic dilations of hypersurfaces in  $\mathbb{R}^3$}

\fancyhead[OR]{\thepage}

\renewcommand{\headrulewidth}{0pt}
\renewcommand{\thefootnote}{\fnsymbol {footnote}}

\title{\textbf{$L^{p} \rightarrow L^{q}$ estimates for maximal functions associated with nonisotropic dilations of hypersurfaces in  $\mathbb{R}^3$}
\footnotetext {This work is supported by Natural Science Foundation of China (No.11601427);
China Postdoctoral Science Foundation (No.2017M613193);  Natural Science Basic Research Plan in Shaanxi Province of China (No.2017JQ1009). }
\footnotetext {{}{2000 \emph{Mathematics Subject
 Classification}: 42B20, 42B25.}}
\footnotetext {{}\emph{Key words and phrases}: Maximal function, nonisotropic dilations, hypersurfaces, $L^{p} \rightarrow L^{q}$ estimate, local smoothing. } } \setcounter{footnote}{0}
\author{
Wenjuan Li, Huiju Wang}

\date{}
\maketitle

\begin{abstract}
The goal of this article is to establish $L^{p} \rightarrow L^{q}$ estimates for maximal functions associated with nonisotropic dilations $\delta_t(x)=(t^{a_1}x_1,t^{a_2}x_2,t^{a_3}x_3)$ of hypersurfaces $(x_{1}, x_{2},\Phi(x_1,x_2))$ in $\mathbb{R}^3$, where the Gaussian curvatures of the hypersurfaces are allowed to vanish.
When $2 \alpha_{2} = \alpha_{3}$, this problem is reduced to  study of the $L^{p} \rightarrow L^{q}$ estimates for maximal functions along the curve $\gamma(x)=(x,x^2(1+\phi(x)))$ and associated dilations $\delta_t(x)=(tx_1,t^2x_2)$. The corresponding maximal function shows features related to the Bourgain circular maximal function, whose $L^{p} \rightarrow L^{q}$ estimate has been considered by [Schlag, JAMS, 1997], [Schlag-Sogge, MRL, 1997] and [Lee, PAMS, 2003]. However, in the study of the maximal function related to the mentioned curve $\gamma(x)$ and associated dilations, we  get the $L^{p} \rightarrow L^{q}$ regularity properties for a family of corresponding Fourier integral operators which fail to satisfy the "cinematic curvature condition" uniformly, which means that classical local smoothing estimates could not be directly applied to our problem. What's more, the $L^{p} \rightarrow L^{q}$ estimates are also new for maximal functions associated with isotropic dilations of hypersurfaces $(x_{1}, x_{2},\Phi(x_1,x_2))$ mentioned before.
\end{abstract}

\section{Introduction}
The spherical maximal function
\begin{equation}
 \sup_{t >0}\left|\int_{\mathcal{S}^{n-1}}f(y-tx)d\sigma(x)\right|,
\end{equation}
where $d\sigma$ is normalized surface measure on the sphere $S^{n-1}$, was first studied by Stein  \cite{stein3} in 1976. The sharp $L^{p} \rightarrow L^{p}$ estimate was established by Stein \cite{stein3} for $p>n/(n-1)$ when $n \ge 3$, and later by  Bougain \cite{JB} for $p> 2$ when $n = 2$. Then many authors turned to the study of generalizations of the spherical maximal function, i.e. the sphere is replaced by a more general smooth hypersurface in $ {\mathbb{R}}^n$. In particular, a natural generalization is to characterize the $L^p$-boundedness properties of the maximal operator associated to hypersurface where the Gaussian curvature at some points is allowed to vanish. Related works can be found in Iosevich \cite{I},  Sogge-Stein \cite{ss}, Sogge \cite{Sogge}, Cowling-Mauceri \cite{cw1,cw2}, Nagel-Seeger-Wainger \cite{nsw},  Iosevich-Sawyer \cite{ios1,ios2}, Iosevich-Sawyer-Seeger \cite{Iss}, Ikromov-Kempe-M\"{u}ller \cite{IKMU} and references therein.

In spite of the $L^{p} \rightarrow L^{p}$ estimate, by modifying the definition of the global circular maximal function, Schlag \cite{WS1} showed that
\begin{equation}
 \sup_{t \in [1,2]}\left|\int_{\mathcal{S}^{1}}f(y-tx)d\sigma(x)\right|
\end{equation}
is actually bounded in the interior of the triangle with vertices $(0,0)$, $(1/2,1/2)$, $(2/5,1/5)$. This result was obtained using "combinatorial method" in \cite{WS1}. Based on some local smoothing estimates, an alternative proof was given by Schlag-Sogge \cite{WS2} later. Schlag-Sogge \cite{WS2} also established $L^{p} \rightarrow L^{q}$ estimates for the local maximal functions of hypersurfaces in $\mathbb{R}^{n}$, but they did not cover hypersurfaces where the Gaussian curvatures at some points are allowed to vanish. It is worth to mention that, using bilinear cone restriction estimate, Lee \cite{SL1} improved the local smoothing estimate in \cite{WS2} and then got endpoint estimate for the local circular maximal function in $\mathbb{R}^{2}$.

What's more, maximal operators defined by averages  over curves or surfaces with nonisotropic dilations have also been extensively considered. In 1970, in the study of a problem related
to Poisson integrals for symmetric spaces, Stein raised the question as to when the operator $\mathcal{M}_{\gamma}$ defined by
$$\mathcal{M}_{\gamma}f(x)=\sup_{h>0}\frac{1}{h}\int_{0}^h|f(x-\gamma(t))|dt,$$
where $\gamma(t)=(A_1t^{a_1},A_2t^{a_2},\cdots,A_nt^{a_n})$ and $A_1,A_2,\cdots,A_n$ are real, $a_i>0$, is bounded on $L^p(\mathbb{R}^n)$.
Nagel, Riviere and Wainger \cite {NRW} showed that the $L^p$-boundedness of $\mathcal{M}_{\gamma}$ holds for  $p>1$ for the special case $\gamma(t)=(t,t^2)$ in $\mathbb{R}^2$ and Stein \cite{stein} for  homogeneous curves in $\mathbb{R}^n$.  For maximal functions $\mathcal{M}$ associated with nonisotropic dilations in higher dimensions, one can see the work by Greenleaf \cite{Greenleaf}, Sogge-Stein \cite{ss}, Iosevich-Sawyer \cite{ios2}, Ikromov-Kempe-M\"{u}ller \cite{IKMU}, Zimmermann \cite{Zimmermann}. More information can be found in \cite{WL} and references therein.

In \cite{WL}, the first author of this paper established $L^p$-estimates for the maximal function related to the hypersurface $(x_{1}, x_{2},\Phi(x_1,x_2))$ in $\mathbb{R}^3$  with associated dilations $\delta_t(x)=(t^{a_1}x_1,t^{a_2}x_2,t^{a_3}x_3)$, $2\alpha_{2} \neq \alpha_{3}$,
 \begin{equation}\label{Eq1.4}
\sup_{t >0}\left|\int_{\mathbb{R}^2}f(y-\delta_t(x_1,x_2,\Phi(x_1,x_2)))\eta(x)dx\right|,
\end{equation}
where $\Phi(x_1,x_2)\in C^{\infty}(\Omega)$  satisfies
\[\partial_2\Phi(0,0)=0,\hspace{0.5cm}\partial_2^2\Phi(0,0)\neq 0.\]
It is clear that the Gaussian curvatures related to hypersurfaces in (\ref{Eq1.4}) are  allowed to  vanish.
In fact, when  dilations satisfy $2a_2 = a_3$, the similar problem has also appeared in the study of  maximal functions associated with the curve $\gamma(x)=(x,x^2(1+\phi(x)))$ and associated dilations $\delta_t(x)=(tx_1,t^2x_2)$, i.e.,
\begin{equation}\label{Eq1.3}
\sup_{t >0}\left|\int_{\mathbb{R}}f(y_1-tx,y_2-t^2x^2\phi(x))\eta(x)dx\right|,
\end{equation}
 where $\eta(x)$ is supported in a sufficiently small neighborhood of the origin. The maximal function shows features related to the Bourgain circular maximal function, which required deep ideas and local smoothing estimates established by Mockenhaupt-Seeger-Sogge  for Fourier integral operators satisfying the so-called "cinematic curvature" condition. However, she observed that  the study of (\ref{Eq1.3}) leads to a family of corresponding Fourier integral operators which fail to satisfy the "cinematic curvature condition" uniformly, which means that classical local smoothing estimates could not be directly applied there. In \cite{WL},  new ideas are established to obtain $L^{4}$-estimate for Fourier integral operators which fail to satisfy the "cinematic curvature condition" uniformly, and finally establish sharp $L^p$-estimates for the maximal function (\ref{Eq1.3}).

By modifying the definition of the  maximal function  defined in (\ref{Eq1.4}) so that the supremum is taken over $t \in [1,2]$, a natural question is to ask the $L^{p} \rightarrow L^{q}$ boundedness of the maximal operators. When $2\alpha_{2} = \alpha_{3}$, we will need to study the  $L^{p} \rightarrow L^{q}$ regularity property of  Fourier integral operators which fail to satisfy the "cinematic curvature condition" uniformly. Furthermore, this research will lead to better understanding of the maximal operator associated with isotropic dilations of  hypersurfaces in $\mathbb{R}^{3}$ where the Gaussian curvatures at some points are allowed to vanish, see Corollary \ref{co1.3} in this article.

We concentrate ourselves to solve this problem in this paper. The corresponding main results will be introduced in subsection \ref{subsection1.1} and \ref{subsection1.2}, respectively.

\subsection{Main theorems for nonisotropic dilations of curves  in $\mathbb{R}^{2}$}\label{subsection1.1}
Let $\phi \in C^{\infty}(I,\mathbb{R})$, where $I$ is a bounded interval containing the origin, and
\begin{equation}\label{phi}
\phi(0)\neq 0; \hspace{0.1cm}\phi^{\prime}(0) \neq 0.
 \end{equation}
We show $L^p \rightarrow L^{q}$ estimates for  maximal functions along curves $(x,x^{2}\phi(x))$ with nonisotropic dilations $(t,t^{2})$.

\begin{thm}\label{planetheorem}
Define the maximal operator
\begin{equation}\label{Eqlocmaximalfunction}
\mathcal{M}f(y):=\sup_{t \in [1,2]}\left|\int_{\mathbb{R}}f(y_1-tx,y_2-t^2x^2\phi(x))\eta(x)dx\right|,
\end{equation}
where $\eta(x)$ is supported in a sufficiently small neighborhood of the origin, $\phi$ satisfies (\ref{phi}). Then for $ \frac{1}{2p} < \frac{1}{q} \le \frac{1}{p} $, $\frac{1}{q}  > \frac{3}{p} -1$,
there exists a constant
$C_{p,q}$ such that the following inequality holds true:
\begin{equation}\label{equ:planem=1}
\|\mathcal{M}f\|_{L^{q}}\leq  C_{p,q} \|f\|_{L^{p}}, \hspace{0.5cm}f\in C_0^{\infty}(\mathbb{R}^2).
\end{equation}
\end{thm}

\begin{rem}
Simple calculations show that $(1/p,1/q)$ in Theorem \ref{planetheorem} actually locate in the triangle with vertices $(0,0)$, $(1/2,1/2)$, $(2/5,1/5)$. This corresponds to the circular maximal function studied by \cite{WS1,WS2}.
\end{rem}

In order to obtain the $L^{p} \rightarrow L^{q}$ estimate  for $\mathcal{M}$, as in the study of $L^{p} \rightarrow L^{p}$ estimate for the global maximal functions defined by
 \begin{equation*}
\sup_{t >0}\left|\int_{\mathbb{R}}f(y_1-tx,y_2-t^2x^2\phi(x))\eta(x)dx\right|,
\end{equation*}
we need to consider a family of corresponding Fourier integral operators which fail to satisfy the "cinematic curvature condition" uniformly, which means that classical local smoothing estimates could not be directly applied to our problem. In order to overcome the above difficulty and finally establish $L^{p} \rightarrow L^{q}$ estimate  for $\mathcal{M}$, we adopt the following strategy.

(i) We break $I$ into dyadic intervals with length $2^{-k}$, $k \ge log(1/|I|)$. By $|I|$ we mean the length of $I$. Meanwhile, we decompose the frequency space of $f$ into $\{\xi \in \mathbb{R}^{2}: |\xi| \leq 1 \}$ and $\{\xi \in \mathbb{R}^{2}: |\xi| \approx 2^{j} \}$, $j \ge 1$. The difficulty lies in the case when $j$ is sufficiently large. It can be observed that for each $j$, $j \gg 1$, the principal curvature for the surfaces related to the  corresponding Fourier integral operators vanishes as $k$ tends to infinity. Therefore, we consider  $k \leq j/2$ and $k > j/2$ respectively.

(ii) For $k > j/2$, no  local smoothing estimates can be established in this case. We only get basically $L^{p} \rightarrow L^{q}$ estimate  for the local maximal operator defined by inequality (\ref{Eq2.28}), see Lemma \ref{lemmaL^2} below. Fortunately, we found Lemma \ref{lemmaL^2} is sufficient for us to finish the proof of Theorem \ref{planetheorem} since $j$ is "small" here.

(iii) For $k \leq j/2$, by Sobolev's embedding Lemma, we are left to consider a class of Fourier integral operators defined by inequality (\ref{estimate:L4}), in which the principal curvature of the related surfaces is not so "small", since upper bound of $k$ is dominated by $j/2$. This phenomenon allows us to obtain $L^{p} \rightarrow L^{q}$ estimate for these Fourier integral operators in Theorem \ref{L^(p,q)therom}, and the proof  of Theorem \ref{L^(p,q)therom}  will be covered in Section 3. In the proof of Theorem \ref{L^(p,q)therom}, we  use  Whitney type decomposition and a bilinear estimate established  by \cite{SL2}.
We remark that a similar but worse result can be obtained by interpolation with $L^{4}$-estimate from Theorem 2.9 in  \cite{WL} and the $L^{1} \rightarrow L^{\infty}$ estimate.

Moreover, the necessary conditions for Theorem \ref{planetheorem} are also considered in Section 4. Unfortunately, we can not show the sharpness of Theorem \ref{planetheorem} for some technical reason.

\subsection{Main theorems for surfaces with one non-vanishing principal curvature in $\mathbb{R}^{3}$}\label{subsection1.2}
Let $\Omega$ be an open neighborhood of the origin. Suppose $\Gamma$ is a hypersurface in $\mathbb{R}^3$ which is parametrized as the graph
of a smooth function $\Phi: \Omega \rightarrow \mathbb{R}$ at the origin, i.e. $\Gamma=\{(x,\Phi(x)), x\in \Omega\subset \mathbb{R}^2\}$. Denote by $\delta_t$ the nonisotropic dilations in $\mathbb{R}^3$ given by
\begin{equation}\label{dilation}
\delta_t(x)=(t^{a_1}x_1,t^{a_2}x_2,t^{a_3}x_3).
\end{equation}
We show $L^{p} \rightarrow L^{q}$ estimates for  maximal functions related to hypersurfaces with at least one non-vanishing principal curvature.


\begin{thm}\label{3nonvanish}
Assume that $\Phi(x_1,x_2)\in C^{\infty}(\Omega)$  satisfies
\begin{equation}\label{conditionnonvani}
\partial_2\Phi(0,0)=0,\hspace{0.5cm}\partial_2^2\Phi(0,0)\neq 0,
\end{equation}
and  $2a_2\neq a_3$. Define the maximal function by
\begin{equation}
\mathcal{M}f(y):=\sup_{t \in [1,2]}\left|\int_{\mathbb{R}^2}f(y-\delta_t(x_1,x_2,\Phi(x_1,x_2)))\eta(x)dx\right|,
\end{equation}
where $\eta$ is supported in a sufficiently small neighborhood $U\subset \Omega$ of the origin. Then for  $ \frac{1}{2p} < \frac{1}{q} \le \frac{1}{p} $, $\frac{1}{q}  > \frac{3}{p} -1$,
there exists a constant
$C_{p,q}$ such that the following inequality holds true:
\begin{equation}\label{equ:dimension3=1}
\|\mathcal{M}f\|_{L^{q}}\leq  C_{p,q} \|f\|_{L^{p}}, \hspace{0.5cm}f\in C_0^{\infty}(\mathbb{R}^3).
\end{equation}
\end{thm}

Notice that when $\alpha_{1}=\alpha_{2} = \alpha_{3}=1$, we get the following corollary.

\begin{cor}\label{co1.3}
Assume that $\Phi(x_1,x_2)\in C^{\infty}(\Omega)$  satisfies inequality (\ref{conditionnonvani}). Define the maximal function by
\begin{equation}
\mathcal{M}f(y):=\sup_{t \in [1,2]}\left|\int_{\mathbb{R}^2}f(y-t(x_1,x_2,\Phi(x_1,x_2)))\eta(x)dx\right|,
\end{equation}
where $\eta$ is supported in a sufficiently small neighborhood $U\subset \Omega$ of the origin. Then for  $ \frac{1}{2p} < \frac{1}{q} \le \frac{1}{p} $, $\frac{1}{q}  > \frac{3}{p} -1$,
there exists a constant
$C_{p,q}$ such that the following inequality holds true:
\begin{equation}\label{equ:dimension3=1}
\|\mathcal{M}f\|_{L^{q}}\leq  C_{p,q} \|f\|_{L^{p}}, \hspace{0.5cm}f\in C_0^{\infty}(\mathbb{R}^3).
\end{equation}
\end{cor}

\begin{rem}
We note that \cite{WS2} established $L^{p} \rightarrow L^{q}$ estimates for the local maximal functions  of hypersurfaces  with non-vanishing Gaussian curvature in $\mathbb{R}^{n}$. While in Corollary \ref{co1.3}, the Gaussian curvatures  of  $(x_{1},x_{2}, \Phi(x_{1},x_{2}))$ are allowed to vanish everywhere.
\end{rem}

We just briefly sketch the proof of Theorem \ref{3nonvanish} here, since most of the details can be found in \cite{WL}, Section 5.1.1. We can always choose non-negative functions $\eta_1$,
$\eta_2$ $\in C_0^{\infty}({\mathbb{R}})$ so that $\eta(x)\leq \eta_1(x_1)\eta_2(x_2)$.
Since
\begin{equation*}
\left|\int_{\mathbb{R}^2}f(y-\delta_t(x_1,x_2,\Phi(x_1,x_2)))\eta(x)dx\right |\leq \int_{\mathbb{R}^2}|f|(y-\delta_t(x_1,x_2,\Phi(x_1,x_2)))\eta_1(x_1)\eta_2(x_2)dx,
\end{equation*}
we may assume $\eta(x)=\eta_1(x_1)\eta_2(x_2)$ and  $f\geq
0$, $a_1=1$. Set $(y_2, y_3)=y'$ and $(\xi_2,\xi_3)=\xi'$. Denote $1+a_2+a_3$ by $Q$ and $(t^{a_2}\xi_2, t^{a_3}\xi_3)$ by
$\delta'_t\xi'$.

First we "freeze" the first variable $x_1$ and  apply the method of stationary phase to curves in $(x_2, x_3)-$ plane, the proof of Theorem \ref{3nonvanish} can be reduced  to estimate the local maximal function defined by
\begin{equation}
\widetilde{\mathcal{M}_{j,loc}^1}f(y):=\sup_{t\in[1,2]}|\widetilde{A_{t,j}^1}f(y)|,
\end{equation}
where
\begin{equation}
\widetilde{A_{t,j}^1}f(y):=\int_{\mathbb{R}}\eta_1(\frac{x_1}{t})\int
_{{\mathbb{R}}^2}e^{i(\xi'\cdot y'-t^{a_3}\xi_3\tilde{\Psi}(\frac{x_1}{t},s))}E_{x_1/t}(\delta'_t\xi')
\beta(2^{-j}|\delta'_t\xi'|)f(y_1-x_1,\widehat{\xi'})d\xi'dx_1,\
\end{equation}
in which  $j\gg 1$, the non-negative function $\beta\in C_0^{\infty}(\mathbb{R})$ such that supp $\beta\subset[1/2,2]$,
\begin{equation*}
E_{x_1}(\delta'_t\xi'):=
\frac{\chi_{x_1}(t^{a_2}\xi_2/t^{a_3}\xi_3)}{(1+|\delta'_t\xi'|)^{1/2}}
A_{x_1}(\delta'_t\xi'),
\end{equation*}
$f(x,\hat{\xi'})$ denotes the partial Fourier transform with respect to the $\xi'$ variables
 and  $\chi_{x_1}$ is a smooth function supported on the set $\{z:|z|<\epsilon_{x_1}\}$, where $\epsilon_{x_1}$ can be controlled by a small positive constant independent of $x_1$. Meanwhile, $A_{x_1}$ is a symbol of order zero.
 The phase function $\tilde{\Psi}(x_1,s):=\Psi(x_1,\psi(x_1,s),s)$,
\begin{equation}
s:=s(\xi',t)=-\frac{t^{a_2}\xi_2}{t^{a_3}\xi_3}, \hspace{0.3cm}\textrm{for} \hspace{0.2cm} \xi_3\neq 0,
\end{equation}
and
\begin{equation}
\Psi(x_1,x_2,s):=-sx_2+\Phi(x_1,x_2).
\end{equation}
where $x_1$ and $s$ are enough small. Here  $\psi$ satisfies
\begin{equation}\label{equation}
\partial_2\Phi(\frac{x_1}{t},\psi(\frac{x_1}{t},s))=s.
\end{equation}

Set
\begin{equation}
Q_{x_1}(y',t,\xi')=\xi'\cdot y'-t^{a_3}\xi_3\tilde{\Psi}(\frac{x_1}{t},s).
\end{equation}
By Sobolev's embedding Lemma, we are left to estimate
\begin{align*}
 \Biggl\|\int_{\mathbb{R}}\eta_1(x_1)\biggl\|\tilde{\rho}(t)\int
_{{\mathbb{R}}^2}e^{iQ_{x_1}(y',t,\xi')}E_{x_1/t}(\delta'_t\xi')\beta(2^{-j}|\delta'_t\xi'|)\\
\quad\quad\quad\times
f(y_1-x_1,\widehat{\xi'})d\xi'
\biggl\|_{L^q([1/2,4]\times \mathbb{R}^2,dtdy')}dx_1
\Biggl\|_{L^q(\mathbb{R}, dy_1)}.
\end{align*}

In order to complete the proof, we mention that Lee \cite{SL2} applied the bilinear method to oscillatory integral operators with variable coefficients, and obtained $L^{p} \rightarrow L^{q}$ regularity properties for a wide class of Fourier integral operators satisfying the "cinematic curvature condition" showed in \cite{mss}.
\begin{thm}(Corollary 1.5 in \cite{SL2})\label{Lee}
Let $\mathcal{F}_{\mu}$ be given by
\begin{equation}
\mathcal{F}_{\mu}f(z) = \int_{\mathbb{R}^{n}}e^{i\phi(z,\xi)}a(z,\xi) \frac{\hat{f}(\xi)}{(1+|\xi|^{2})^{\mu/2}}d\xi, \quad z=(x,t).
\end{equation}
Suppose supp $a(\cdot,\xi)$ is contained in a fixed compact set and suppose that $\phi(\cdot, \xi)$ is a homogeneous function of degree one. For all $(z,\xi) \in $ supp $a$, $\phi$ satisfies
\[\text{rank } \partial^{2}_{z\xi}\phi = n,\]
and
\[\text{rank }\partial^{2}_{\xi\xi} \langle \partial_{z}\phi, \theta \rangle= n-1,\]
provided $\theta \in \mathcal{S}^{n}$ is the unique direction for which $\nabla_{\xi} \langle \partial_{z}\phi, \theta \rangle= 0$, also all non-zero eigenvalues of $\partial^{2}_{\xi\xi} \langle \partial_{z}\phi, \theta \rangle$ have the same sign. Then for $2(n^{2}+2n-1)/(n^{2}-1) \le q \le \infty$, $(n+1)/q \le (n-1)(1-1/p)$, $q \ge p(n+3)/(n+1)$,
\begin{equation}
\|\mathcal{F}_{\mu}f\|_{L^{q}}\leq  C \|f\|_{L^{p}}
\end{equation}
provided $\mu > 1/p - (n+1)/q + (n-1)/2$.
\end{thm}
 By a similar argument as in \cite{WL}, Section 5.1.1, we can restrict $y^{^{\prime}}$ in a fixed compact set. It is easy to check that the phase function $Q_{x_1}(y',t,\xi')$ satisfies the  conditions in Theorem \ref{Lee}, the inner $L^{q}$ norm (in which we "freeze" $x_{1}$) can be dominated by Theorem \ref{Lee}. Notice that in our case, the Fourier support of $f$ is contained in the annular $\{\xi^{\prime} \in \mathbb{R}^{2}: |\xi^{\prime}| \sim 2^{j}\}$. Then the proof of Theorem \ref{3nonvanish} will be finished by Young's inequality.

\begin{thm}\label{theocurvanishi}Let $\phi\in C^{\infty}(I)$, where $I$ is a bounded interval containing the origin.
Define the maximal function by
\begin{equation}
\mathcal{M}f(y):=\sup_{t \in [1,2]}\left|\int_{\mathbb{R}^2}f(y-\delta_t(x_1,x_2,x_2^{2}\phi(x_2)))\eta(x)dx\right|,
\end{equation}
where $\eta$ is supported in a sufficiently small neighborhood $U$ of the origin.
Assume that $\phi$ satisfies (\ref{phi}),
and $2a_2=a_3$.  Then for  $ \frac{1}{2p} < \frac{1}{q} \le \frac{1}{p} $, $\frac{1}{q}  > \frac{3}{p} -1$,
there exists a constant
$C_{p,q}$ such that the following inequality holds true:
\begin{equation}
\|\mathcal{M}f\|_{L^{q}}\leq  C_{p,q} \|f\|_{L^{p}}, \hspace{0.5cm}f\in C_0^{\infty}(\mathbb{R}^3).
\end{equation}
\end{thm}

For the proof of Theorem \ref{theocurvanishi}, we follow the idea in the proof of Theorem \ref{3nonvanish}. First we "freeze" the first variable $x_1$ and apply the method of stationary phase to curves in $(x_2, x_3)-$ plane, then by Sobolev's embedding Lemma, we can reduce to apply $L^p \rightarrow L^{q}$ estimates for certain Fourier integral operators. However,  the phase function in the operators here no longer satisfy the cinematic curvature  condition. Therefore, Theorem \ref{Lee} is not available. Instead, we apply Theorem \ref{L^(p,q)therom} below to finish the proof of Theorem \ref{theocurvanishi}. The details are also omitted since most of them can be found in \cite{WL}, Section 5.1.2.

\section{Proof of Theorem \ref{planetheorem}}
We choose $B>0$ very small and  $\tilde{\rho}\in C_0^{\infty}(\mathbb{R})$ such that  supp $\tilde{\rho}\subset\{x:B/2\leq|x|\leq 2B\}$  and $\sum_k\tilde{\rho}(2^kx)=1$ for $x\in \mathbb{R}$.

Put
\begin{align*}
A_tf(y):&=\int f(y_1-tx,y_2-t^2x^2\phi(x))\eta(x)dx\\
&=\sum_k\int f(y_1-tx,y_2-t^2x^2\phi(x))\tilde{\rho}(2^kx)\eta(x)dx=\sum_kA_t^kf(y),
\end{align*}
where
\begin{equation}
A_t^kf(y):=\int f(y_1-tx,y_2-t^2x^2\phi(x))\tilde{\rho}(2^kx)\eta(x)dx.
\end{equation}

Since $\eta$ is supported in a sufficiently small neighborhood of the origin, we only need to consider $k>0$ sufficiently large.

Considering  isometric operator on $L^p(\mathbb{R}^2)$ defined by $T_kf(x_1,x_2)=2^{3k/p}f(2^kx_1,2^{2k}x_2)$, one can compute that
\begin{equation}
T_k^{-1}A_t^kT_kf(y)=2^{-k}\int f(y_1-tx,y_2-t^2x^2\phi(\frac{x}{2^k}))\tilde{\rho}(x)\eta(2^{-k}x)dx.
\end{equation}

Then it suffices to prove the following estimate
\begin{equation}\label{Target}
\sum_k2^{3k(\frac{1}{p} - \frac{1}{q})-k}\left\|\sup_{t \in [1,2]}|\widetilde{A_t^k}|\right\|_{L^p\rightarrow L^q}\leq C_{p,q}
\end{equation}
for $p,q$ as in Theorem \ref{planetheorem}, where
\begin{equation}
\widetilde{A_t^k}f(y):=\int f(y_1-tx,y_2-t^2x^2\phi(\frac{x}{2^k}))\tilde{\rho}(x)\eta(2^{-k}x)dx.
\end{equation}

By means of the Fourier inversion formula, we have
\begin{align*}
\widetilde{A_t^k}f(y)&=\frac{1}{(2\pi)^2}\int_{{\mathbb{R}}^2}e^{i\xi\cdot y}\int_{\mathbb{R}}e^{-i(t\xi_1x+t^2\xi_2x^2\phi(\frac{x}{2^k}))}\tilde{\rho}(x)\eta(2^{-k}x)dx\hat{f}(\xi)d\xi
\\
&=\frac{1}{(2\pi)^2}\int_{{\mathbb{R}}^2}e^{i\xi\cdot y}\widehat{d\mu_k}(\delta_t\xi)\hat{f}(\xi)d\xi,
\end{align*}
where
\begin{equation}\label{dmu_k}
\widehat{d\mu_k}(\xi):=\int_{\mathbb{R}}e^{-i(\xi_1x+\xi_2x^2\phi(\frac{x}{2^k}))}\tilde{\rho}(x)\eta(2^{-k}x)dx.
\end{equation}

We choose a non-negative function $\beta\in C_0^{\infty}(\mathbb{R})$ such that supp $\beta\subset[1/2,2]$ and $\sum_{j\in\mathbb{Z}}\beta(2^{-j}r)=1$ for $r>0$. Define the dyadic operators
\begin{equation}
\widetilde{A_{t,j}^k}f(y)=\frac{1}{(2\pi)^2}\int_{{\mathbb{R}}^2}e^{i\xi\cdot y}\widehat{d\mu_k}(\delta_t\xi)\beta(2^{-j}|\delta_t\xi|)\hat{f}(\xi)d\xi,
\end{equation}
and denote by $\widetilde{\mathcal{M}_{j}^k}$ the corresponding maximal operator.  Now we have that
\begin{equation*}
\sup_{t \in [1,2]}|\widetilde{A_t^k}f(y)|\leq \widetilde{\mathcal{M}^{k,0}}f(y)+\sum_{j\geq 1}\widetilde{\mathcal{M}_{j}^k}f(y), \hspace{0.2cm}\textmd{for}\hspace{0.2cm}y\in \mathbb{R}^2,
\end{equation*}
where
\begin{equation}
\widetilde{\mathcal{M}^{k,0}}f(y):=\sup_{t \in [1,2]}|\sum_{j\leq 0}\widetilde{A_{t,j}^k}f(y)|.
\end{equation}

We will often use the following method of stationary phase.
\begin{lem} (Theorem 1.2.1 in \cite{sogge2})\label{lem:Lemma1}
Let S be a smooth hypersurface in $\mathbb{R}^n$ with non-vanishing Gaussian curvature and $d\mu$ be the Lebesgue measure on $S$. Then,
\begin{equation}
|\widehat{d\mu}(\xi)|\leq C(1+|\xi|)^{-\frac{n-1}{2}}.
\end{equation}
\end{lem}

Note that $\widetilde{\mathcal{M}^{k,0}}f(y)=\sup_{t \in [1,2]}|f*K_{\delta_{t^{-1}}}(y)|$, where $K_{\delta_{t^{-1}}}(x)=t^{-3}K(\frac{x_1}{t},\frac{x_2}{t^2})$ and
\begin{equation}
K(y):=\int_{{\mathbb{R}}^2}e^{i\xi\cdot y}\widehat{d\mu_k}(\xi)\rho(|\xi|)d\xi,
 \end{equation}
where $\rho\in C_0^{\infty}(\mathbb{R})$ is supported in $[0,2]$.  Since  $\phi$ satisfies (\ref{phi}), supp $\tilde{\rho}\subset \{x:B/2\leq |x|\leq 2B\}$,
then Lemma \ref{lem:Lemma1} implies that for a multi-index $\alpha$,
\begin{equation}
\left|\left(\frac{\partial}{\partial\xi}\right)^{\alpha}\widehat{d\mu_k}(\xi)\right|\leq C_{B,\alpha}(1+|\xi|)^{-1/2}.
\end{equation}

By integration by parts, we obtain that
\begin{equation}
|K(y)|\leq C_N (1+|y|)^{-N}.
\end{equation}
Then by $q \ge p$ and Young's inequality, we have
\begin{align}
\|\widetilde{\mathcal{M}^{k,0}}f\|_{L^{q}} &= \|\sup_{t \in [1,2]}|f*K_{\delta_{t^{-1}}}|\|_{L^{q}} \nonumber\\
&\le \biggl\|\frac{C_{N}}{(1+|\cdot|)^{N}} * |f|\biggl\|_{L^{q}} \nonumber\\
&\lesssim \|f\|_{L^{p}}.
\end{align}
 So it suffices to prove that
\begin{equation}\label{loc}
\sum_k2^{3k(\frac{1}{p} - \frac{1}{q})-k}\sum_{j\geq 1}\|\widetilde{\mathcal{M}_{j}^k}\|_{L^p\rightarrow L^q} \leq C_{p,q}.
\end{equation}

In order to get (\ref{loc}), first we will consider
\begin{equation}\label{mainestim}
\widehat{d\mu_k}(\delta_t\xi)=\int_{\mathbb{R}}e^{-it^2\xi_2(-sx+x^2\phi(\delta x))}\tilde{\rho}(x)\eta(\delta x)dx,
\end{equation}
where $2^{-k}=\delta$ and
\begin{equation}
s:=s(\xi,t)=-\frac{\xi_1}{t\xi_2}, \hspace{0.3cm}\textrm{for} \hspace{0.2cm} \xi_2\neq 0.
\end{equation}
If $\xi_2=0$, then
\begin{equation*}
|\widehat{d\mu_k}(\delta_t\xi)|=|(\eta(\delta\cdot)\tilde{\rho})^{\wedge}(t\xi_1)|\leq \frac{C_N'}{(1+|t\xi_1|)^N}=\frac{C_N'}{(1+|\delta_t\xi|)^N},
\end{equation*}
and for multi-index $\alpha$,
\begin{equation*}
|D_{\xi}^{\alpha}\widehat{d\mu_k}(\delta_t\xi)|=|D_{\xi}^{\alpha}(\eta(\delta\cdot)\tilde{\rho})^{\wedge}(t\xi_1)|\leq \frac{C_{\alpha,N}}{(1+|\delta_t\xi|)^N}.
\end{equation*}
Since $t\approx 1$, we will put the case $\xi_2=0$ in $B_k$ of the following (\ref{station}).

Put
\begin{equation}
\Phi(s,x,\delta)=-sx+x^2\phi(\delta x),
\end{equation}
then we have
\begin{equation*}
\partial_x\Phi(s,x,\delta)=-s+2x\phi(\delta x)+x^2\delta\phi'(\delta x)
\end{equation*}
and
\begin{equation*}
\partial_x^2\Phi(s,x,\delta)=2\phi(\delta x)+4x\delta\phi'(\delta x)+x^2\delta^2\phi''(\delta x).
\end{equation*}

Since $k$ is sufficiently large and $\phi(0)\neq 0$, then the implicit function theorem implies that there exists a smooth solution $x_c=\tilde{q}(s,\delta)$ of the equation $\partial_x\Phi(s,x,\delta)=0$. For the sake of simplicity, we may assume $\phi(0)=1/2$. By Taylor's expansion, the phase function can be written as
\begin{equation}\label{phase function}
-t^2\xi_2\tilde{\Phi}(s,\delta)=\frac{\xi_1^2}{2\xi_2}+\delta\frac{\xi_1^3}{t\xi_2^2}\phi'(0) + \delta^2 \xi_{2}R(\frac{\xi_{1}}{\xi_{2}},t,\delta),
\end{equation}
$-t^2\xi_2\tilde{\Phi}(s,\delta)$ can be considered as a small perturbation of  $\frac{\xi_1^2}{2\xi_2}+\delta\frac{\xi_1^3}{t\xi_2^2}\phi'(0)$.

By applying the method of stationary phase, we have
 \begin{equation}\label{station}
\widehat{d\mu_k}(\delta_t\xi)=e^{-it^2\xi_2\tilde{\Phi}(s,\delta)}\chi_k(\frac{\xi_1}{t\xi_2})
\frac{A_k(\delta_t\xi)}{(1+|\delta_t\xi|)^{1/2}}+B_k(\delta_t\xi),
\end{equation}
where $\chi_k$ is a smooth function supported in the interval $[c_k,\tilde{c}_k]$, for certain non-zero positive constants $c_1\leq c_k, \tilde{c_k}\leq c_2$ depending only on $k$. $\{A_k(\delta_t\xi)\}_k$ is contained in a bounded subset of symbols of order zero. More precisely, for arbitrary $t \in [1,2]$,
 \begin{equation}\label{symbol}
 |D_{\xi}^{\alpha}A_k(\delta_t\xi)|\leq C_{\alpha}(1+|\xi|)^{-\alpha},
\end{equation}
where $C_{\alpha}$ is independent of $k$ and $t$. Furthermore, $B_k$ is a remainder term and satisfies for arbitrary $t\in [1,2]$,
\begin{equation}\label{symbol2}
 |D_{\xi}^{\alpha}B_k(\delta_t\xi)|\leq C_{\alpha,N}(1+|\xi|)^{-N},
\end{equation}
where $C_{\alpha,N}$ are admissible constants and again do not depend on $k$ and $t$.

First, let us consider the remainder part of (\ref{loc}). Set
\begin{equation}
M_{j}^{k,0}f(y):=\sup_{t\in[1,2]}\left|\frac{1}{(2\pi)^2}\int_{{\mathbb{R}}^2}e^{i\xi\cdot y}B_k(\delta_t\xi)\beta(2^{-j}|\delta_t\xi|)\hat{f}(\xi)d\xi \right|.
\end{equation}
By (\ref{symbol2}) and integration by parts, it is easy to get $|(B_k\beta(2^{-j}\cdot))^{\vee}(x)|\leq C_N 2^{-jN}(1+|x|)^{-N}$. Therefore,
\begin{align}
M_{j}^{k,0}f(y) &\le \sup_{t\in[1,2]}\frac{C_{N}2^{-jN}}{(2\pi)^2t^{3}}\int_{{\mathbb{R}}^2} \frac{|f(x)|}{(1+|\delta_{t^{-1}}(y-x)|)^{N}} dx \nonumber\\
 &\lesssim 2^{-jN}\int_{{\mathbb{R}}^2} \frac{|f(x)|}{(1+|y-x|)^{N}} dx. \nonumber
\end{align}
Young's inequality and the fact that $3(\frac{1}{p} - \frac{1}{q}) < 1$ imply (\ref{Target}) for remainder part of (\ref{loc}).

Put
\begin{equation}
A_{t,j}^kf(y):=\frac{1}{(2\pi)^2}\int_{{\mathbb{R}}^2}e^{i(\xi \cdot y-t^2\xi_2\tilde{\Phi}(s,\delta))}\chi_k(\frac{\xi_1}{t\xi_2})
\frac{A_k(\delta_t\xi)}{(1+|\delta_t\xi|)^{1/2}}\beta(2^{-j}|\delta_t \xi|)\hat{f}(\xi)d\xi.
\end{equation}
Denote by $M_{j}^{k,1}$ the corresponding maximal operator over $[1,2]$. It remains to prove that
\begin{equation} \label{object}
\sum_k 2^{3k(\frac{1}{p} - \frac{1}{q})-k}\sum_{j\geq 1} \|M_{j}^{k,1}\|_{L^p\rightarrow L^q}\leq C_{p,q,N}.
\end{equation}

Since $\tilde{\Phi}(s,\delta)$ is homogeneous of degree zero in $\xi$ and $\frac{\xi_1}{\xi_2}\approx 1$,
then
$$\left|\nabla_{\xi}[\xi \cdot (y-x)-t^2\xi_2\tilde{\Phi}(s,\delta)]\right|\geq C|y-x|$$
provided $|y-x|\geq L$, where $L$ is very large and determined by $c_1$, $c_2$ and $\|\phi\|_{\infty\hspace{0.1cm} (I)}$. By integration by parts, we will see that the kernel of the operator $A_{t,j}^k$ is dominated by $2^{-jN}\mathcal{O}(|y-x|^{-N})$ if $|y-x|\geq L$.  From now on, we will restrict our view on the situation
\begin{equation}\label{case}
|y-x|\leq L.
\end{equation}

Let $B_i(L)$ be a ball with center $i$ and radius $L$.  It is easy to show that
\begin{equation}\label{local}
\sup\{\|M_{j}^{k,1}f\|_{L^q}:\|f\|_{L^p}=1, \textmd{supp}\hspace{0.2cm} f\subset B_0(L)\}\leq C_{p}2^{-j\tilde{\epsilon}_1(p,q)}2^{k\tilde{\epsilon}_2(p,q)}
\end{equation}
implies that
\begin{equation}\label{local2}
\|M_{j}^{k,1}\|_{L^p\rightarrow L^q}\leq C_{p,q}2^{-j\tilde{\epsilon}_1(p,q)}2^{k\tilde{\epsilon}_2(p,q)},
\end{equation}
where $C_{p,q}$ depends on $p$, $q$, $c_1$, $c_2$ and $\|\phi\|_{\infty\hspace{0.1cm} (I)}$, and $\tilde{\epsilon}_1(p,q)$, $\tilde{\epsilon}_2(p,q)>0$. Then in order to prove inequality (\ref{loc}), it suffices to prove inequality (\ref{local}).

Now we observe inequality (\ref{local}),  together with the assumption (\ref{case}),  we can choose $\rho_1\in C_0^{\infty}(\mathbb{R}^2\times[\frac{1}{2},4])$ such that (\ref{local}) will follow from that
\begin{equation}
\|\widetilde{M_{j}^{k,1}}\|_{L^p\rightarrow L^q}\leq C_{p}2^{-j\tilde{\epsilon}_1(p,q)}2^{k\tilde{\epsilon}_2(p,q)},
\end{equation}
where
\begin{equation}\label{Eq2.28}
\widetilde{M_{j}^{k,1}}f(y):=\sup_{t\in[1,2]}\left|\rho_1(y,t)A_{t,j}^kf(y)\right|.
\end{equation}

\begin{lem}\label{lemmaL^2}
Suppose that $p \le q \le 2p$, $pq \ge 3q-p$. Then for any $\epsilon >0$, we have
\begin{equation}\label{p-q}
\|\widetilde{M_{j}^{k,1}}f\|_{L^{q}}\leq C_{p,q} 2^{j(1+\epsilon)(1/p-1/q)-(j\wedge k)/q}\|f\|_{L^p}.
\end{equation}
\end{lem}

\begin{proof}
In order to prove inequality (\ref{p-q}), we first show
\begin{equation}\label{1-infty}
\|\widetilde{M_{j}^{k,1}}f\|_{L^ \infty }\leq C_{\epsilon} 2^{j(1+\epsilon)}\|f\|_{L^1}.
\end{equation}
We introduce the angular decomposition of the set $\{\xi \in \mathbb{R}^{2}: \frac{\xi_1}{\xi_2}\approx 1\}$.  For each positive integer $j$, we consider a roughly equally spaced set of points with grid length $2^{-j/2}$ on the unit circle $S^1$; that is, we fix a collection $\{\kappa_{j}^{\nu}\}_{\nu}$ of real numbers, that satisfy:

$(a)$ $|\kappa_j^{\nu}-\kappa_j^{\nu'}|\geq 2^{-j/2}$, if $\nu\neq\nu'$;

$(b)$ if $\xi\in \{\xi \in \mathbb{R}^{2}: \frac{\xi_1}{\xi_2}\approx 1\}$, then there exists a $\kappa_j^{\nu}$ so that $\biggl|\frac{\xi_1}{\xi_2}-\kappa_j^{\nu}\biggl|<2^{-j/2}$.

Let $\Gamma_j^{\nu}$ denote the corresponding cone in the $\xi$-space
\begin{equation*}
\Gamma_j^{\nu}=\{\xi \in \mathbb{R}^{2}: \biggl|\frac{\xi_{1}}{\xi_{2}}-\kappa_j^{\nu}\biggl|\leq2\cdot 2^{-j/2}\}.
\end{equation*}
We can construct an associated partition of unity:
$\chi_j^{\nu}$ is  homogeneous of degree zero in $\xi$ and supported in $\Gamma_j^{\nu}$, with
\begin{equation}\label{anghomofunct1}
\sum_{\nu}\chi_j^{\nu}(\xi)=1 \hspace{0.5cm}\textrm{for}\hspace{0.2cm}\textrm{ all}\hspace{0.2cm} \xi\in \{\xi \in \mathbb{R}^{2}: \frac{\xi_1}{\xi_2}\approx 1\} \hspace{0.2cm}\textrm{and}\hspace{0.2cm} \textrm{all}\hspace{0.2cm} j,
\end{equation}
and
\begin{equation}\label{anghomofunct2}
|\partial_{\xi}^{\alpha}\chi_j^{\nu}(\xi)|\leq A_{\alpha}2^{|\alpha|j/2}|\xi|^{-|\alpha|}.
\end{equation}

Hence, in order to establish (\ref{1-infty}), notice that
\begin{align}
\|\widetilde{M_{j}^{k,1}}f\|_{L^{\infty}} &= \biggl \|\sup_{t \in [1,2]}\biggl |\int_{\mathbb{R}^{2}} K_{t}(y,x)f(x)dx \biggl |\biggl \|_{L^{\infty}} \nonumber\\
 &\le \biggl \|\sup_{t \in [1,2]}\int_{\mathbb{R}^{2}} \sum_{\nu}{ |K_t^{\nu}(y,x)|}|f(x)|dx \biggl \|_{L^{\infty}},
\end{align}
where
\begin{equation}\label{ker1-1}
K_t(y,x)=\rho_1(y,t)\int_{{\mathbb{R}}^2}e^{i(\xi \cdot (y-x)-t^2\xi_2\tilde{\Phi}(s,\delta))}\widetilde{A_k}(\xi,t)\beta(2^{-j}|\delta_t \xi|)d\xi,
\end{equation}
and
\begin{equation}\label{ker1}
K_t^{\nu}(y,x)=\rho_1(y,t)\int_{{\mathbb{R}}^2}e^{i(\xi \cdot (y-x)-t^2\xi_2\tilde{\Phi}(s,\delta))}\widetilde{A_k}(\xi,t)\beta(2^{-j}|\delta_t \xi|)\chi_j^{\nu}(\xi)d\xi,
\end{equation}
and $\widetilde{A_k}(\xi,t)=\chi_k(\frac{\xi_1}{t\xi_2})
\frac{A_k(\delta_t\xi)}{(1+|\delta_t\xi|)^{1/2}}$. If we can show that
for fixed  $y$, $t$ and $\epsilon >0$, we have
\begin{equation}\label{sum bound}
\sum_{\nu}{ |K_t^{\nu}(y,x)|} \le 2^{(1+\epsilon)j},
\end{equation}
uniformly for $x \in \mathbb{R}^{2}$, then inequality (\ref{1-infty}) follows.

In fact, it is not hard to check that for each $\nu$,
\begin{align}\label{kernalL^infty}
|K_t^{\nu}(y,x)| \leq 2^{j} \frac {C _{N}}{(1+ 2^{\frac{j}{2}}|x_{1}-c_{1}(y,t,\kappa_{j}^{\nu},\delta)|  + 2^{j}|x_{2} +\kappa_{j}^{\nu}x_{1}-c_{2}(y,t,\kappa_{j}^{\nu},\delta)|   )^{N}}, \nonumber
\end{align}
where $C$ does not depend on $t$, $j$, $k$ and $\nu$,
\[c_{1}(y,t,\kappa_{j}^{\nu},\delta) = \kappa_{j}^{\nu}+y_{1}+\delta (\partial_{1}\bar{R})(\kappa_{j}^{\nu},t,\delta),\]
\[c_{2}(y,t,\kappa_{j}^{\nu},\delta) =y_{2} +\kappa_{j}^{\nu}y_{1} +\frac{1}{2}(\kappa_{j}^{\nu})^{2}+\delta \bar{R}(\kappa_{j}^{\nu},t,\delta),\]
here $\bar{R}(\frac{\xi_{1}}{\xi_{2}},t,\delta) =\frac{\xi_1^3}{t\xi_2^2}\phi'(0) + \delta \xi_{2}R(\frac{\xi_{1}}{\xi_{2}},t,\delta)$.

For fixed $y$, $t$, $j$, $k$, if one of
\[|x_{1}-c_{1}(y,t,\kappa_{j}^{\nu},\delta)| \ge 2^{-\frac{j}{2}+\epsilon j},\]
and
\[|x_{2} + \kappa_{j}^{\nu}x_{1}-c_{2}(y,t,\kappa_{j}^{\nu},\delta)| \ge 2^{-j+\epsilon j}\]
holds true for some $\epsilon >0$, then we have
\[|K_t^{\nu}(y,x)| \leq C_{N}2^{-\epsilon Nj}.\]
 Inequality (\ref{sum bound}) follows  since $N$ can be sufficiently large. Therefore, we only need to consider the case when $(x_{1}, x_{2})$ satisfies
\[|x_{1}-c_{1}(y,t,\kappa_{j}^{\nu},\delta)| \le 2^{-\frac{j}{2}+\epsilon j},\]
\[|x_{2} + \kappa_{j}^{\nu}x_{1}-c_{2}(y,t,\kappa_{j}^{\nu},\delta)| \le 2^{-j+\epsilon j}.\]
It is obvious that for fixed $y$, $t$, $j$, $k$, if $\nu\neq\nu'$,
\[|c_{1}(y,t,\kappa_{j}^{\nu},\delta) -c_{1}(y,t,\kappa_{j}^{\nu'},\delta)| \ge 2^{-j/2} ,\]
which implies inequality (\ref{sum bound}).

By Lemma 2.7 in \cite{WL}, we have
\begin{equation}\label{p-p}
\|\widetilde{M_{j}^{k,1}}f\|_{L^p}\leq C_{p} 2^{- (j \wedge k)/p}\|f\|_{L^p}, \hspace{0.5cm}2\leq p\leq \infty.
\end{equation}
Then inequality (\ref{p-q}) follows from  the M. Riesz interpolation theorem between (\ref{1-infty}) and (\ref{p-p}).
\end{proof}

Now we split the set of $j$ into two parts $j>2k$ and $j\leq 2k$. When $j\leq 2k$,
by inequality (\ref{p-q}), if $p < q \leq 2p$,  $pq \ge 3q-p$, $p > 5/2$, then for any $\epsilon >0$, we get
\begin{align}\label{jle2k}
\sum_k 2^{3k(\frac{1}{p} - \frac{1}{q})-k} \sum_{j\leq 2k}\|\widetilde{M_{j}^{k,1}}f\|_{L^q} &\leq C_{p,q} \sum_k 2^{5k(\frac{1}{p} - \frac{1}{q})-k +\epsilon k} \|f\|_{L^p} \nonumber\\
&\leq C_{p,q} \sum_k 2^{\frac{5}{2p}k-k + \epsilon k} \|f\|_{L^p} \nonumber\\
&\leq C_{p,q} \|f\|_{L^p}.
\end{align}
Next we consider the case $j>2k$, in which the following lemma is required.

\begin{lem} (Theorem 2.4.2 in \cite{sogge2})\label{lem:Lemma3}
Suppose that $F$ is $C^1(\mathbb{R})$. Then if $p>1$ and $1/p+1/p'=1$,
\begin{equation*}
\sup_{\lambda}|F(\lambda)|^p\leq |F(0)|^p+p\biggl(\int|F(\lambda)|^pd\lambda\biggl)^{1/p'}\biggl(\int|F'(\lambda)|^pd\lambda\biggl)^{1/p}.
\end{equation*}
\end{lem}

For $j> 2k$, by Lemma \ref{lem:Lemma3},
\begin{equation}\label{well}
\begin{aligned}
&\|\widetilde{M_{j}^{k,1}}f\|_{L^q}\\
&\leq C
2^{\frac{j}{q} -\frac{k}{q} -\frac{j}{2}}\biggl(\int_{{\mathbb{R}}^2}\int^4_{1/2}\biggl|\rho_1(y,t)\int_{{\mathbb{R}}^2}e^{i(\xi \cdot y-t^2\xi_2\tilde{\Phi}(s,\delta))}2^{j/2}\widetilde{A_k}(\xi,t)\beta(2^{-j}|\delta_t \xi|)\hat{f}(\xi)d\xi\biggl|^qdtdy \biggl)^{
(1-1/q)\times 1/q}\\
& \quad \times
\biggl(\int_{{\mathbb{R}}^2}\int^4_{1/2}\biggl|\frac{\partial}{\partial t}(\rho_1(y,t)\int_{{\mathbb{R}}^2}e^{i(\xi \cdot y-t^2\xi_2\tilde{\Phi}(s,\delta))}2^{k-j/2}\widetilde{A_k}(\xi,t)\beta(2^{-j}|\delta_t \xi|)\hat{f}(\xi)d\xi)\biggl|^qdtdy \biggl)^{1/q^{2}}.
\end{aligned}
\end{equation}

In order to simplify the notations, we choose $\tilde{\chi}\in C_0^{\infty}([c_1,c_2])$ so that  $\tilde{\chi}(\frac{\xi_1}{\xi_2})\chi(\frac{\xi_1}{t\xi_2})=\chi(\frac{\xi_1}{t\xi_2})$
 for arbitrary $t\in [1/2,4]$ and $k$ sufficiently large. In a similar way we choose $\rho_0\in C_0^{\infty}((-10,10))$ such that $\rho_0(|\xi|)\beta(|\delta_t \xi|)=\beta(|\delta_t \xi|)$ for arbitrary $t\in [1/2,4]$. Furthermore, since $A_k$ satisfies (\ref{symbol}), if  $a(\xi,t):=2^{j/2}\widetilde{A}_k(\xi,t)\beta(2^{-j}|\delta_t \xi|)$ for $k$ sufficiently large, then $a(\xi,t)$ is a symbol of order zero, i.e. for any $t\in [1/2,4]$, $\alpha\in {\mathbb{N}}^2$,
\begin{equation}\label{symbola}
\left|\left(\frac{\partial}{\partial \xi}\right)^{\alpha}a(\xi,t)\right|\leq C_{\alpha}(1+|\xi|)^{-|\alpha|}.
\end{equation}

Assume that we have obtained the following theorem, the proof will be found later.
\begin{thm}\label{L^(p,q)therom}
For all $j > 2k$, and $p,q$ satisfying $\frac{1}{2p}< \frac{1}{q} \le  \frac{3}{5p}$, $\frac{3}{q} \le 1-\frac{1}{p} $, $\frac{1}{q} \ge \frac{1}{2} -\frac{1}{p} $, we have
\begin{equation}
\biggl(\int_{{\mathbb{R}}^2}\int^4_{1/2}|\tilde{F}_{j}^{k}f(y,t)|^qdtdy \biggl)^{1/q}
\leq C 2^{\frac{k}{2}(1-\frac{1}{p}-\frac{1}{q})} 2^{(\frac{1}{2} -\frac{3}{q}+ \frac{1}{p} +\epsilon)j} \|f\|_{L^p({\mathbb{R}}^2)},\hspace{0.2cm}\textmd{some }\hspace{0.2cm}\epsilon >0,
\end{equation}
where
\begin{equation}\label{estimate:L4}
 \tilde{F}_{j}^{k}f(y,t)=\rho_1(y,t)\int_{{\mathbb{R}}^2}e^{i(\xi \cdot y-t^2\xi_2\tilde{\Phi}(s,\delta))}a(\xi,t)\rho_0(2^{-j}|\xi|)\tilde{\chi}(\frac{\xi_1}{\xi_2})\hat{f}(\xi)d\xi.
\end{equation}
\end{thm}

By (\ref{well}), we obtain
\begin{align*}
\|\widetilde{M_{j}^{k,1}}f\|_{L^q({\mathbb{R}}^2)}&\leq C 2^{\frac{j}{q} -\frac{k}{q} -\frac{j}{2}}2^{\frac{k}{2}(1-\frac{1}{p}-\frac{1}{q})}  2^{(\frac{1}{2} -\frac{3}{q}+ \frac{1}{p} +\epsilon)j} \|f\|_{L^p({\mathbb{R}}^2)}\\
&=C 2^{\frac{k}{2}(1-\frac{1}{p}-\frac{3}{q})} 2^{(\frac{1}{p} -\frac{2}{q}  +\epsilon)j} \|f\|_{L^p({\mathbb{R}}^2)}
\end{align*}
for $p,q$ satisfying conditions in Theorem \ref{L^(p,q)therom}. Then we have
\begin{align*}
\sum_k 2^{3k(\frac{1}{p} - \frac{1}{q})-k}\sum_{j>2k}\|\widetilde{M_{j}^{k,1}}\|_{L^p\rightarrow L^q}\leq C\sum_k2^{\frac{k}{2}(\frac{5}{p} - \frac{9}{q}-1)}\sum_{j>2k}2^{(\frac{1}{p} -\frac{2}{q}+  \epsilon)j}
\leq C_{p,q}.
\end{align*}
This and inequality (\ref{jle2k}) imply for $p,q$ satisfying $\frac{1}{2p} <\frac{1}{p} \le \frac{3}{5p}$, $\frac{3}{q} \le 1-\frac{1}{p} $, $\frac{1}{q} \ge \frac{1}{2} -\frac{1}{p} $, we have
\begin{equation}\label{equ:planem=1small}
\|\mathcal{M}f\|_{L^{q}}\leq  C_{p,q} \|f\|_{L^{p}}.
\end{equation}
By Theorem 1.1 in \cite{WL}, for each $2 \le p \le \infty$,
\begin{equation}\label{diag}
\|\mathcal{M}f\|_{L^{p}}\leq  C_{p} \|f\|_{L^{p}}, \hspace{0.5cm}f\in C_0^{\infty}(\mathbb{R}^2).
\end{equation}
Combining inequality (\ref{equ:planem=1small}) with inequality (\ref{diag}), then Theorem \ref{planetheorem} follows from the M. Riesz interpolation theorem.

\section{Proof of Theorem \ref{L^(p,q)therom}}
In order to prove Theorem \ref{L^(p,q)therom}, we will use  Whitney type decomposition and a bilinear estimate established by \cite{SL2}. By rescaling, we turn to estimate
\begin{equation}\label{estimate:L4+}
\mathcal{F}_{j}^{k}g(y,t)=\rho_1(y,t)\int_{{\mathbb{R}}^2}e^{i2^{j}(\xi \cdot y-t^2\xi_2\tilde{\Phi}(s,\delta))}a(2^{j}\xi,t)\rho_0(|\xi|)\tilde{\chi}(\frac{\xi_1}{\xi_2})\hat{g}(\xi)d\xi,
\end{equation}
since $\tilde{\Phi}(s,\delta)$ is homogeneous of degree one in $\xi$. Notice that by finite decomposition, we may assume that $\tilde{\chi}\in C_0^{\infty}([c_1,c_2])$, where
\[|c_1-c_2| \le \epsilon_{0}\]
with $\epsilon_{0}$ sufficiently small. By Whitney decomposition, we have
\[(\mathcal{F}_{j}^{k}g)^{2}= \sum_{l: (j-k)/2 \le l \le log1/\epsilon_{0}} \sum_{dist(C_{\theta}^{l}, C_{\theta^{\prime}}^{l}) \sim 2^{-l}} \mathcal{F}_{j}^{k}g_{\theta}^{l} \mathcal{F}_{j}^{k}g_{\theta^{\prime}}^{l},\]
where $\{C_{\theta}^{l}\}_{l: (j-k)/2 \le l \le log1/\epsilon_{0}}$ are sectors with center direction $\theta$ and angular $2^{-l}$,
\[\widehat{g_{\theta}^{l}}(\xi) = \hat{g}(\xi)\chi_{C_{\theta}^{l}}(\frac{\xi_{1}}{\xi_{2}}).\]
Here we abuse the notation by saying $dist(C_{\theta}^{l}, C_{\theta^{\prime}}^{l}) \sim 2^{-l}$ to mean $dist(C_{\theta}^{l}, C_{\theta^{\prime}}^{l}) \le 2^{-l}$ when $l=(j-k)/2$. By the similar argument as in \cite{SL2}, we can establish the orthogonality that for $q \ge 4$,
 \begin{align}\label{Eq3.2}
\|\mathcal{F}_{j}^{k}g\|_{_{L^{q}(\mathbb{R}^{3})}}
&\lesssim 2^{j \epsilon} \biggl\{\sum_{l: (j-k)/2 \le l \le log1/\epsilon_{0}} \biggl(\sum_{dist(C_{\theta}^{l}, C_{\theta^{\prime}}^{l}) \sim 2^{-l}} \| \mathcal{F}_{j}^{k}g_{\theta}^{l} \mathcal{F}_{j}^{k}g_{\theta^{\prime}}^{l}\|_{_{L^{q/2}(\mathbb{R}^{3})}}^{(q/2)^{\prime}} \biggl)^{1/(q/2)^{\prime}} \biggl \}^{1/2}.
\end{align}

It is sufficient to prove the following two lemmas.

\begin{lem}\label{main}
For each $l$: $2^{l} \ll 2^{(j-k)/2}$, we have
\begin{equation*}
\| \mathcal{F}_{j}^{k}g_{\theta}^{l} \mathcal{F}_{j}^{k}g_{\theta^{\prime}}^{l}\|_{_{L^{q/p}(\mathbb{R}^{3})}} \leq C 2^{-l+\epsilon j} 2^{(k + 3l -3j)\frac{p}{q}} \|g_{\theta}^{l}\|_{L^2}\|g_{\theta^{\prime}}^{l}\|_{L^2}.
\end{equation*}
\end{lem}

\begin{lem}\label{basic}
For each $l$: $2^{l} \approx 2^{(j-k)/2}$, we have
\begin{equation*}
\| \mathcal{F}_{j}^{k}g_{\theta}^{l} \mathcal{F}_{j}^{k}g_{\theta^{\prime}}^{l}\|_{_{L^{q/p}(\mathbb{R}^{3})}} \leq C 2^{-2j\frac{p}{q}} 2^{-\frac{j-k}{2}(1-\frac{p}{q})} \|g_{\theta}^{l}\|_{L^2}\|g_{\theta^{\prime}}^{l}\|_{L^2}.
\end{equation*}
\end{lem}

Indeed, it can be observed from Lemma 2.7 in \cite{WL} that for each $l: (j-k)/2 \le l \le log1/\epsilon_{0}$,
\begin{equation*}
\| \mathcal{F}_{j}^{k}g_{\theta}^{l} \mathcal{F}_{j}^{k}g_{\theta^{\prime}}^{l}\|_{_{L^{\infty}(\mathbb{R}^{3})}} \leq C 2^{-2l+j}  \|g_{\theta}^{l}\|_{L^\infty}\|g_{\theta^{\prime}}^{l}\|_{L^\infty}.
\end{equation*}
If Lemma \ref{main} and Lemma \ref{basic} hold true, then by the M. Riesz interpolation theorem, for each $l$: $2^{l} \ll 2^{(j-k)/2}$, we obtain
\begin{equation}\label{Eq3.3}
\| \mathcal{F}_{j}^{k}g_{\theta}^{l} \mathcal{F}_{j}^{k}g_{\theta^{\prime}}^{l}\|_{_{L^{q/2}(\mathbb{R}^{3})}} \leq C 2^{\frac{2k}{q}} 2^{2(\frac{1}{2} -\frac{3}{q}- \frac{1}{p} +\epsilon)j} 2^{2(-1+\frac{3}{q}+ \frac{1}{p} )l} \|g_{\theta}^{l}\|_{L^p}\|g_{\theta^{\prime}}^{l}\|_{L^p}.
\end{equation}
And for each $l$: $2^{l} \approx 2^{(j-k)/2}$, we obtain
\begin{equation}\label{Eq3.4}
\| \mathcal{F}_{j}^{k}g_{\theta}^{l} \mathcal{F}_{j}^{k}g_{\theta^{\prime}}^{l}\|_{_{L^{q/2}(\mathbb{R}^{3})}} \leq C 2^{k(1-\frac{1}{p}-\frac{1}{q})} 2^{(-\frac{3}{q}- \frac{1}{p} )j}  \|g_{\theta}^{l}\|_{L^p}\|g_{\theta^{\prime}}^{l}\|_{L^p}.
\end{equation}
Since $\frac{3}{q} \le 1-\frac{1}{p} $, inequalities (\ref{Eq3.3}) and (\ref{Eq3.4}) imply that for each $l: (j-k)/2 \le l \le log1/\epsilon_{0}$,
\begin{equation}\label{Eq3.5}
\| \mathcal{F}_{j}^{k}g_{\theta}^{l} \mathcal{F}_{j}^{k}g_{\theta^{\prime}}^{l}\|_{_{L^{q/2}(\mathbb{R}^{3})}} \leq C 2^{k(1-\frac{1}{p}-\frac{1}{q})}  2^{2(\frac{1}{2} -\frac{3}{q}- \frac{1}{p} +\epsilon)j} \|g_{\theta}^{l}\|_{L^p}\|g_{\theta^{\prime}}^{l}\|_{L^p}.
\end{equation}
Since for each $l$, $C_{\theta}^{l}, C_{\theta^{\prime}}^{l}$ are almost disjoint, and the assumption that $\frac{1}{q} \ge \frac{1}{2} -\frac{1}{p} $, we get
\begin{align}\label{Eq3.6}
\sum_{dist(C_{\theta}^{l}, C_{\theta^{\prime}}^{l}) \sim 2^{-l}} \| g_{\theta}^{l}\|_{_{L^{p}(\mathbb{R}^{3})}}^{2(q/2)^{\prime}} \lesssim \|g\|_{L^{p}}^{2(q/2)^{\prime}}; \quad
\sum_{dist(C_{\theta}^{l}, C_{\theta^{\prime}}^{l}) \sim 2^{-l}} \| g_{\theta^{\prime}}^{l}\|_{_{L^{p}(\mathbb{R}^{3})}}^{2(q/2)^{\prime}} \lesssim \|g\|_{L^{p}}^{2(q/2)^{\prime}}.
\end{align}
It follows from inequalities (\ref{Eq3.2}), (\ref{Eq3.5}), H\"{o}lder's inequality and (\ref{Eq3.6}) that
\begin{align}
\|\mathcal{F}_{j}^{k}g\|_{_{L^{q}(\mathbb{R}^{3})}} \leq C 2^{\frac{k}{2}(1-\frac{1}{p}-\frac{1}{q})} 2^{(\frac{1}{2} -\frac{3}{q}- \frac{1}{p} +\epsilon)j} \|g\|_{L^p({\mathbb{R}}^2)},
\end{align}
which implies Theorem \ref{L^(p,q)therom} by rescaling.

Let's turn to prove Lemma \ref{main} and Lemma \ref{basic}. However, the proof of Lemma \ref{basic} is quite trivial. In fact, for each $l,j,k$, we have the following basic estimates
\begin{align}
\|\mathcal{F}_{j}^{k}g_{\theta}^{l}\|_{_{L^{2}(\mathbb{R}^{3})}} \leq C 2^{-j} \|\widehat{g_{\theta}^{l}}\|_{L^2({\mathbb{R}}^2)},
\end{align}
\begin{align}
\|\mathcal{F}_{j}^{k}g_{\theta}^{l}\|_{_{L^{\infty}(\mathbb{R}^{3})}} \leq C 2^{-l} \|\widehat{g_{\theta}^{l}}\|_{L^\infty({\mathbb{R}}^2)},
\end{align}
\begin{align}
\|\mathcal{F}_{j}^{k}g_{\theta}^{l}\|_{_{L^{\infty}(\mathbb{R}^{3})}} \leq \|\widehat{g_{\theta}^{l}}\|_{L^1({\mathbb{R}}^2)},
\end{align}
then Lemma \ref{basic} can be obtained by  the M. Riesz interpolation theorem and H\"{o}lder's inequality. Therefore, we will prove Lemma \ref{main} in the rest of this  section.

We first do some reductions. Put $\theta=(\theta_{1}, \theta_{2})$, $\theta^{\prime}=(\theta_{1}^{\prime}, \theta_{2}^{\prime})$, and
\[\kappa = \frac{1}{2} \biggl(\frac{\theta_{1}}{\theta_{2}}+\frac{\theta_{1}^{\prime}}{\theta_{2}^{\prime}}), \:\ \kappa \thickapprox 1,\]
then  $C_{\theta}^{l}, C_{\theta^{\prime}}^{l}$ can be written as
\[C_{\theta}^{l}:=\biggl\{(\xi_{1},\xi_{2}): \frac{1}{2}2^{-l} \le \frac{\xi_{1}}{\xi_{2}} - \kappa \le \frac{3}{2} 2^{-l} \biggl\},\]
\[C_{\theta^{\prime}}^{l}:=\biggl\{(\xi_{1},\xi_{2}):-\frac{3}{2}2^{-l} \le \frac{\xi_{1}}{\xi_{2}} - \kappa \le -\frac{1}{2}2^{-l-1} \biggl\}.\]
By changes of variables,
\begin{equation}\label{Eq3.11}
\xi_{1} = 2^{-l}\eta_{1} + \kappa \eta_{2}, \quad \xi_{2} = \eta_{2},
\end{equation}
we get
\[\mathcal{C}_{1}:=\biggl\{(\eta_{1},\eta_{2}): \frac{1}{2} \le \frac{\eta_{1}}{\eta_{2}} \le \frac{3}{2}  \biggl\},\]
\[\mathcal{C}_{2}:=\biggl\{(\eta_{1},\eta_{2}): -\frac{3}{2} \le \frac{\eta_{1}}{\eta_{2}} \le -\frac{1}{2}  \biggl\},\]
and reduce Lemma \ref{main}  to the following lemma.
\begin{lem}\label{last}
For each $l$: $2^{l} \ll 2^{(j-k)/2}$, and any function $g_{i}$ with supp $\hat{g_{i}} \subset \mathcal{C}_{i}$,
$i=1,2$, we have
\begin{equation}\label{Eqglobal}
\| \mathcal{G}_{j,l}^{k}g_{1} \mathcal{G}_{j,l}^{k}g_{2}\|_{_{L^{q/p}(\mathbb{R}^{3})}} \leq C 2^{\epsilon j} 2^{(k + 3l -3j)\frac{p}{q}} \|g_{1}\|_{L^2}\|g_{2}\|_{L^2},
\end{equation}
where
\[\mathcal{G}_{j,l}^{k}g_{i}(x,t)=\tilde{\rho}_1(x,t)\int_{{\mathbb{R}}^2}e^{i2^{j}\Psi(x,t,\eta,\delta,2^{-l})} \tilde{a}(2^{j}\eta,t)\tilde{\rho_0}(|\eta|) \widehat{g_{i}}(\eta)d\eta\]
with the phase function
\[\Psi(x,t,\eta,\delta,2^{-l}) = 2^{-l}x_{1}\eta_{1}+x_{2}\eta_{2} +  2^{-2l-k}\frac{3\phi^{\prime}(0)}{t}\frac{\eta_{1}^{2}}{\eta_{2}}  +2^{-3l-k}\frac{\phi^{\prime}(0)}{t}\frac{\eta_{1}^{3}}{\eta_{2}^{2}}+2^{-2l-2k}\eta_2\tilde{R}(\frac{\eta_{1}}{\eta_{2}},t,2^{-l},\delta),\]
\[\tilde{R}(\frac{\eta_{1}}{\eta_{2}},t,2^{-l},\delta) =\frac{\eta_{1}^{2}}{2\eta_{2}^{2}}(\partial _{1}^{2}R)(\kappa,t,\delta) +2^{-l}\frac{\eta_{1}^{3}}{6\eta_{2}^{3}}(\partial _{1}^{3}R)(\kappa,t,\delta),\]
and $\tilde{\rho_{1}}(x,t) \in C_{0}^{\infty}(\mathbb{R}^{3})$,
\[\tilde{a}(2^{j}\eta,t) = a(2^{j}\tau(\eta),t),\quad \tilde{\rho_0}(|\eta|)=\rho_0(|\tau(\eta)|), \quad \tau: (\eta_{1},\eta_{2}) \rightarrow (2^{-l}\eta_{1} + \kappa \eta_{2},\eta_{2}).\]
\end{lem}

More concretely, if Lemma \ref{last} holds true, under the transformation (\ref{Eq3.11}),
denote
\[\hat{h_{i}}(\eta) =e^{i2^{j} \frac{(2^{-l}\eta_{1}+\kappa \eta_{2})^{2}}{2\eta_{2}} } 2^{-l}\chi_{i}(\frac{\eta_{1}}{\eta_{2}} ) \widehat{g \circ \tau^{-1}}(\eta), \quad i=1,2, \]
by Plancherel,
\begin{align}
\| \mathcal{F}_{j}^{k}g_{\theta}^{l} \mathcal{F}_{j}^{k}g_{\theta^{\prime}}^{l}\|_{_{L^{q/p}(\mathbb{R}^{3})}}
&\leq  \| \mathcal{G}_{j,l}^{k}h_{1} \mathcal{G}_{j,l}^{k} h_{2}  \|_{_{L^{q/p}(\mathbb{R}^{3})}} \nonumber\\
&\leq C 2^{\epsilon j} 2^{(k + 3l -3j)\frac{p}{q}} \Pi_{i-1}^{2} \biggl\|2^{-l} e^{i2^{j} \frac{(2^{-l}\eta_{1}+\kappa \eta_{2})^{2}}{2\eta_{2}} } \chi_{i}(\frac{\eta_{1}}{\eta_{2}}  ) \widehat{g \circ \tau^{-1}}(\eta) \biggl\|_{L^2} \nonumber\\
&= C 2^{-l+\epsilon j} 2^{(k + 3l -3j)\frac{p}{q}} \|g_{\theta}^{l}\|_{L^2}\|g_{\theta^{\prime}}^{l}\|_{L^2}, \nonumber
\end{align}
where
\[\chi_{1}(\frac{\eta_{1}}{\eta_{2}}) = \chi_{C_{\theta}^{l}}(2^{-l}\frac{\eta_{1}}{\eta_{2}} + \kappa), \quad \chi_{2}(\frac{\eta_{1}}{\eta_{2}}) = \chi_{C_{\theta^{\prime}}^{l}}(2^{-l}\frac{\eta_{1}}{\eta_{2}} + \kappa).\]

Now we turn to the proof of Lemma \ref{last}. Notice that in the phase function
$\Psi(x,t,\eta,\delta,2^{-l})$, the terms $2^{-3l-k}\frac{\phi^{\prime}(0)}{t}\frac{\eta_{1}^{3}}{\eta_{2}^{2}}+2^{-2l-2k}\eta_2\tilde{R}(\frac{\eta_{1}}{\eta_{2}},t,2^{-l},\delta)$
can be considered as a small perturbation of  $2^{-2l-k}\frac{3\phi^{\prime}(0)}{t}\frac{\eta_{1}^{2}}{\eta_{2}} $, since $k$ and $l$ can be sufficiently large. Let $Q_{0}=[-2^{-l-k}, 2^{-l-k}] \times [-2^{-2l-k}, 2^{-2l-k}]$,
then
\begin{equation}\label{Eqlocal1}
\| \mathcal{G}_{j,l}^{k}g_{1} \mathcal{G}_{j,l}^{k}g_{2}\|_{_{L^{q/p}(Q_{0})}} \leq  2^{(-2k - 3l)\frac{p}{q}} \| Tg_{1} Tg_{2}\|_{_{L^{q/p}(Q(0,1))}},
\end{equation}
where $Q(0,1)$ denotes the unit square and
\[Tg_{i}(x,t)= \int_{{\mathbb{R}}^2}e^{i2^{j-k-2l}\tilde{\Psi}(x,t,\eta,\delta,2^{-l})} \tilde{a}(2^{j}\eta,t)\tilde{\rho_0}(|\eta|) \widehat{g_{i}}(\eta)d\eta\]
with the phase function
\[\tilde{\Psi}(x,t,\eta,\delta,2^{-l}) = x_{1}\eta_{1}+x_{2}\eta_{2}+\frac{3\phi^{\prime}(0)}{t}\frac{\eta_{1}^{2}}{\eta_{2}}  +2^{-l}\frac{\phi^{\prime}(0)}{t}\frac{\eta_{1}^{3}}{\eta_{2}^{2}}+2^{-k}\eta_2\tilde{R}(\frac{\eta_{1}}{\eta_{2}},t,2^{-l},\delta).\]
It follows from  Theorem 1.2 in \cite{SL2} that
\begin{equation}\label{Eq3.14}
 \| Tg_{1} Tg_{2}\|_{_{L^{q/p}(Q(0.1))}} \leq C 2^{(3k + 6l-3j + \epsilon j)\frac{p}{q}} \|g_{1}\|_{L^2}\|g_{2}\|_{L^2}.
\end{equation}
Inequalities (\ref{Eqlocal1}) and (\ref{Eq3.14}) imply that
\begin{equation}\label{Eqlocal}
\| \mathcal{G}_{j,l}^{k}g_{1} \mathcal{G}_{j,l}^{k}g_{2}\|_{_{L^{q/p}(Q_{0})}} \leq C 2^{\epsilon j} 2^{(k + 3l -3j)\frac{p}{q}} \|g_{1}\|_{L^2}\|g_{2}\|_{L^2},
\end{equation}
and by translation invariance in the $x-$plane, this inequality actually holds for any $2^{-l-k} \times 2^{-2l-k}$ rectangles. Now we will show that the global estimate (\ref{Eqglobal}) follows from the local estimate (\ref{Eqlocal}). We have
\begin{align}
\mathcal{G}_{j,l}^{k}g_{i}(x,t)=\tilde{\rho}_1(x,t) \int_{\mathbb{R}^{2}} \int_{{\mathbb{R}}^2}e^{i2^{j}\Psi(x,t,\eta,\delta,2^{-l})-iz \cdot \eta} \tilde{a}(2^{j}\eta,t)\tilde{\rho_0}(|\eta|) \chi_{i}(\frac{\eta_{1}}{\eta_{2}})  d\eta g_{i}(z)dz. \nonumber
\end{align}
Denote
\begin{align}
K_{i}(x,z,t)=  \int_{{\mathbb{R}}^2}e^{i2^{j}\Psi(x,t,\eta,\delta,2^{-l})-iz \cdot \eta} \tilde{a}(2^{j}\eta,t)\tilde{\rho_0}(|\eta|) \chi_{i}(\frac{\eta_{1}}{\eta_{2}}) d\eta. \nonumber
\end{align}
Integration by parts show that for each $i=1,2$,
\begin{align}
|K_i(x,z,t)| \leq  \frac {C _{N}}{\biggl(1+ 2^{j-l}\biggl | | x_{1}-\frac{z_{1}}{2^{j-l}}| +\mathcal{O}(2^{-l-k}) \biggl| + 2^{j}\biggl | | x_{1}-\frac{z_{1}}{2^{j}}|+\mathcal{O}(2^{-2l-k}) \biggl|   \biggl)^{N}}, \nonumber
\end{align}
which implies that $(x_{1},x_{2})$ can be considered roughly in a $2^{-l-k} \times 2^{-2l-k}$ rectangle with the center $(\frac{z_{1}}{2^{j-l}},\frac{z_{2}}{2^{j}})$. Therefore
\begin{align}\label{sum}
&\mathcal{G}_{j,l}^{k}g_{1}(x,t)\mathcal{G}_{j,l}^{k}g_{2}(x,t) \nonumber\\
&=\tilde{\rho}_1(x,t)^{2} \sum_{Q,Q^{\prime}} \Pi_{cQ \cap cQ^{\prime}}(x) \int_{\mathbb{R}^{2}} K_{1}(x,z,t) \Pi_{Q}(\frac{z_{1}}{2^{j-l}},\frac{z_{2}}{2^{j}}) g_{1}(z)dz \nonumber\\
&\quad \times \int_{\mathbb{R}^{2}} K_{2}(x,z,t) \Pi_{Q^{\prime}}(\frac{z_{1}}{2^{j-l}},\frac{z_{2}}{2^{j}}) g_{2}(z)dz
 +C_{2N} 2^{-jN} 2^{-kN} \tilde{\rho}_1(x,t)^{2} \|g_{1}\|_{L^2}\|g_{2}\|_{L^2} \nonumber\\
 &=\tilde{\rho}_1(x,t)^{2} \sum_{Q,Q^{\prime}} \Pi_{cQ \cap cQ^{\prime}}(x) \mathcal{G}_{j,l}^{k} \biggl( \Pi_{Q}(\frac{z_{1}}{2^{j-l}},\frac{z_{2}}{2^{j}})g_{1}(z) \biggl)(x,t) \mathcal{G}_{j,l}^{k} \biggl( \Pi_{Q^{\prime}}(\frac{z_{1}}{2^{j-l}},\frac{z_{2}}{2^{j}}) g_{2}(z) \biggl)(x,t) \nonumber\\
 &\quad +C_{2N} 2^{-jN} 2^{-kN} \tilde{\rho}_1(x,t)^{2} \|g_{1}\|_{L^2}\|g_{2}\|_{L^2},
\end{align}
in which $Q,Q^{\prime}$ are $2^{-l-k} \times 2^{-2l-k}$ rectangles, by $\Pi_{Q}$ and  $\Pi_{Q^{\prime}}$ we mean the restriction on $Q,Q^{\prime}$, respectively. Notice that by uncertainty principle and the assumption that $2^{l} \ll 2^{(j-k)/2}$, the Fourier transform of $\Pi_{Q}(\frac{z_{1}}{2^{j-l}},\frac{z_{2}}{2^{j}}) g_{1}(z)$ is supported in a sufficiently small neighborhood of $\mathcal{C}_{1}$. For the same reason, the Fourier transform of $\Pi_{Q^{\prime}}(\frac{z_{1}}{2^{j-l}},\frac{z_{2}}{2^{j}}) g_{2}(z)$ is supported in a sufficiently small neighborhood of $\mathcal{C}_{2}$. Then inequalities (\ref{Eqlocal}) and (\ref{sum}) imply that
\begin{align}
&\| \mathcal{G}_{j,l}^{k}g_{1} \mathcal{G}_{j,l}^{k}g_{2}\|_{_{L^{q/p}( \mathbb{R}^{3} )}} \nonumber\\
&\leq C 2^{(k + 3l -3j+\epsilon j)\frac{p}{q}}\sum_{Q,Q^{\prime}: cQ \cap cQ^{\prime} \neq \emptyset} \biggl\|\Pi_{Q}(\frac{z_{1}}{2^{j-l}},\frac{z_{2}}{2^{j}}) g_{1}(z) \biggl\|_{L^{2}} \biggl\|\Pi_{Q^{\prime}}(\frac{z_{1}}{2^{j-l}},\frac{z_{2}}{2^{j}})g_{2}(z) \biggl\|_{L^{2}}  \nonumber\\
 &\quad +C_{2N} 2^{-jN} 2^{-kN} \tilde{\rho}_1(x,t)^{2} \|g_{1}\|_{L^2}\|g_{2}\|_{L^2} \nonumber\\
 & \leq C 2^{\epsilon j} 2^{(k + 3l -3j+\epsilon j)\frac{p}{q}} \|g_{1}\|_{L^2}\|g_{2}\|_{L^2}.
\end{align}
This completes the proof of Lemma \ref{last}.

\section{Necessary Condition}
Schlag proved  that for the circular maximal function in $2$-dimensional case
 \begin{equation}
\mathcal{M_{C}}f(y) = \sup_{t \in [1,2]}\left|\int_{\mathcal{S}^{1}}f(y-tx)d\sigma(x)\right|,
\end{equation}
conditions $(C1)$-$(C3)$ below are necessary for  the $L^{p} \rightarrow L^{q}$ estimates to hold:

(C1) $\frac{1}{q} \le \frac{1}{p}$;

(C2) $ \frac{1}{q} \ge \frac{1}{2p}$;

(C3) $ \frac{1}{q} \ge \frac{3}{p}-1$.
\\
The constructions of counterexamples in \cite{WS1, WS2} for (C1)-(C2)  can be generalized to maximal functions along smooth curves of various types. However, this is not true for (C3).

In fact, after a change of coordinates and variable, the  circular maximal function may be written as
\begin{equation}
\mathcal{M_{C}}f(y) \approx \sup_{t \in [1,2]}\left|\int_{\mathbb{R}}f(y_1-x,y_2-\sqrt{t^{2}-x^{2}})\eta(x)dx\right|,
\end{equation}
for some cut-off function $\eta(x)$ supported around the origin. It is clear that for each $t \in [1,2]$, the curves
\[\Gamma_{\mathcal{C}_{t}}(x)= \sqrt{t^{2}-x^{2}}\]
pass through $(0,t)$ with uniform normal direction $(0,1)$. Therefore, when $f$ is a characteristic function on $[-\delta^{1/2}, \delta^{1/2}] \times [-\delta, \delta]$, $\delta \ll 1$, then fix $y \in [-\delta^{1/2}, \delta^{1/2}] \times [1, 2]$, there exists $t_{y} \in [1,2]$ such that
\[(y_1-x,y_2-\sqrt{t_{y}^{2}-x^{2}}) \in [-\delta^{1/2}, \delta^{1/2}] \times [-\delta, \delta]\]
for all $x \in [-\delta^{1/2}, \delta^{1/2}]$.
Therefore
\[ \left|\int_{-\delta^{1/2}}^{\delta^{1/2}}f(y_1-x,y_2-\sqrt{t_{y}^{2}-x^{2}})\eta(x)dx\right| \geq \delta^{1/2}.\]
Then (C3) is established.

The necessary conditions in \cite{WS1, WS2} can be generalized to some  classes  of maximal functions, such as
\begin{equation}
\mathcal{M}_{\mathcal{P}_{1}}f(y) := \sup_{t \in [1,2]}\left|\int_{\mathbb{R}}f(y_1-tx,y_2-t^{2}(\phi(x)+1))\eta(x)dx\right|,
\end{equation}
where $\phi(x)$ satisfies $\phi(0) =\phi^{\prime}(0)= 0, \hspace{0.1cm}\phi^{\prime \prime}(0) \neq 0$. Specially, it is valid for $\phi(x)=x^{2}$.
It can be checked  that $\mathcal{M}_{\mathcal{P}_{1}}$ can not be bounded from $L^{p}(\mathbb{R}^{2})$ to $L^{q}(\mathbb{R}^{2})$ if one of (C1)-(C3) holds true.
But this generalization depends heavily on the geometric property of the curve. For evidence, consider the maximal function along paraboloid passing through the origin,
\begin{equation}\label{Eq4.4}
\mathcal{M}_{\mathcal{P}_{2}}f(y) := \sup_{t \in [1,2]}\left|\int_{\mathbb{R}}f(y_1-tx,y_2-t^{2}x^{2})\eta(x)dx\right|.
\end{equation}
It can be proved that the $L^{p} \rightarrow L^{q}$ estimate holds true  provided that $ \frac{1}{2p} < \frac{1}{q} \le \frac{1}{p} $, $\frac{1}{q}  > \frac{2}{p} -1$. A simple calculation shows that the line $\frac{1}{q} = \frac{3}{p} -1$  crosses the triangle determined by $ \frac{1}{2p} < \frac{1}{q} \le \frac{1}{p} $ and $\frac{1}{q}  > \frac{2}{p} -1$. Therefore, (C3) is not a necessary condition for the $L^{p} \rightarrow L^{q}$ estimate to hold for the maximal function in (\ref{Eq4.4}).

In Theorem \ref{planetheorem}, we studied maximal function along paraboloid passing through the origin  with small perturbations. We do not know how to prove the sharpness of Theorem \ref{planetheorem}. Instead we show the following theorem.

\begin{thm}\label{negative}
Set

$(C1^{\star})$ $ \frac{1}{q} \leq \frac{1}{p}$;

$(C2^{\star})$ $ \frac{1}{q} \geq \frac{1}{2p}$;

$(C3^{\star})$ $ \frac{1}{q} \ge \frac{2}{p}-1$. \\
Then $(C1^{\star})$, $(C2^{\star})$ and $(C3^{\star})$ are necessary for inequality (\ref{equ:planem=1}).
\end{thm}
\begin{proof}
We only prove the necessity of $(C3^{\star})$. Without loss of generality, we may choose
\[\phi(x)=x.\]
Assume that  $\epsilon_{0}$ is a small positive constant and
\[\eta(x)=1, \quad  x \in [0,\epsilon_{0}].\]
 Take $0< \delta \ll \epsilon_{0}$.

Then we select $\{t_{i}\}_{i=1}^{N} \subset [1,2]$, $N=100^{-1}\epsilon_{0}^{3}\delta^{-1}$, such that for each $1 \leq i \leq N-1$,
\[|t_{i}-t_{i+1}| = 100\epsilon_{0}^{-3}\delta.\]
Choose
\[Q_{i}=[\frac{\epsilon_{0}}{2}-\frac{\delta}{2}, \frac{\epsilon_{0}}{2} +\frac{\delta}{2}] \times [\frac{\epsilon_{0}^{2}}{4} + \frac{\epsilon_{0}^{3}}{8}\frac{1}{t_{i}}-\frac{\delta}{2}, \frac{\epsilon_{0}^{2}}{4} + \frac{\epsilon_{0}^{3}}{8}\frac{1}{t_{i}} +\frac{\delta}{2} ].\]
By the construction  of $\{t_{i}\}_{i=1}^{N}$, $\{Q_{i}\}_{i=1}^{N}$ are disjoint cubes, so
\begin{equation}\label{Eq4.5}
\biggl| \bigcup_{i=1}^{N}Q_{i} \biggl| \approx \delta.
\end{equation}

When $f$ is a characteristic function on $[-\delta, \delta] \times [-\delta, \delta]$, we have
\begin{equation}\label{Eq4.6}
\|f\|_{L^{p}(\mathbb{R}^{2})} \sim \delta^{\frac{2}{p}}.
\end{equation}
And for fixed $y \in \cup_{i=1}^{N}Q_{i} $, there exists $t_{i}$ such that
\[\biggl( y_1-x,y_2-x^{2}-\frac{x^{3}}{t_{i}} \biggl) \in [-\delta, \delta] \times [-\delta, \delta]\]
for all $x \in [\frac{\epsilon_{0}}{2}-\frac{\delta}{2}, \frac{\epsilon_{0}}{2} +\frac{\delta}{2}]$,
which implies
\[ f \biggl( y_1-x,y_2-x^{2}-\frac{x^{3}}{t_{i}} \biggl) = 1,\]
and then
\begin{equation}\label{Eq4.7}
 \left|\int_{\frac{\epsilon_{0}}{2}-\frac{\delta}{2}}^{\frac{\epsilon_{0}}{2}+\frac{\delta}{2}}f \biggl( y_1-x,y_2-x^{2}-\frac{x^{3}}{t_{i}} \biggl)\eta(x)dx\right| \geq \delta.
 \end{equation}

Since $t \in [1,2]$, changes of variables and inequality (\ref{Eq4.7}) show that for every $y \in \cup_{i=1}^{N}Q_{i} $,
\begin{equation}\label{Eq4.8}
\biggl |\mathcal{M}f(y) \biggl| \geq \delta.
\end{equation}
Inequalities (\ref{Eq4.5}) and (\ref{Eq4.8}) imply that
\begin{equation}\label{4.9}
\|\mathcal{M}f\|_{L^{q}(\mathbb{R}^{2})} \geq  \delta^{1+\frac{1}{q}},
\end{equation}
combining this with inequality (\ref{Eq4.6}), we get
\begin{equation}\label{4.9}
\frac{\|\mathcal{M}f\|_{L^{q}(\mathbb{R}^{2})}}{\|f\|_{L^{p}(\mathbb{R}^{2})}} \geq  \delta^{1+\frac{1}{q}-\frac{2}{p}}.
\end{equation}
Then  the necessity of $(C3^{\star})$ is established since $\delta$ can be sufficiently small.
\end{proof}


\begin{flushleft}
\vspace{0.3cm}\textsc{Wenjuan Li\\School of Mathematics and Statistics\\Northwest Polytechnical University\\710129\\Xi'an, People's Republic of China}

\vspace{0.3cm}\textsc{Huiju Wang\\School of Mathematics Sciences\\University of Chinese Academy of Sciences\\100049\\Beijing, People's Republic of China}

\end{flushleft}

\end{document}